\documentclass[11pt]{amsart}
\usepackage{amsfonts}
\usepackage{amscd}
\usepackage{amssymb}
\usepackage{graphics}

\usepackage{amsmath}

\usepackage{hyperref}
\hypersetup{colorlinks,citecolor=blue}

\newcommand{\BZ}{{\mathbb{Z}}}
\newcommand{\BN}{{\mathbb{N}}}
\newcommand{\BR}{{\mathbb{R}}}

\newcommand{\gC}{\Gamma}

\newcommand{\ga}{\alpha}

\newcommand{\ti}[1]{\tilde{#1}}

\newcommand{\vol}{\text{vol}}

\newcommand{\SL}{\text{SL}}

\newcommand{\SO}{\text{SO}}

\newtheorem{prop}{Proposition}[section]
\newtheorem{thm}[prop]{Theorem}
\newtheorem{lem}[prop]{Lemma}
\newtheorem{cor}[prop]{Corollary}
\newtheorem{conj}[prop]{Conjecture}
\theoremstyle{definition}

\newtheorem{defn}[prop]{Definition}

\catcode`\@=11
\long\def\@savemarbox#1#2{\global\setbox#1\vtop{\hsize\marginparwidth 
  \@parboxrestore\tiny\raggedright #2}}
\catcode`\@=12

\begin{document}
\author{Tsachik Gelander and Paul Vollrath}

\address{Mathematics and Computer Science\\
Weizmann Institute\\
Rechovot 76100, Israel\\}
\email{tsachik.gelander@gmail.com, paul.vollrath@weizmann.ac.il}
\thanks{Supported by ISF-Moked grant 2095/15 and by ERC advanced grant 101052954}

\date{\today}

\title{Bounds on systoles and homotopy complexity}

\maketitle

\begin{abstract}
The main purpose of this note is to prove a lower bound on the systole of compact arithmetic manifolds in terms of the volume. The proof relies on Prasad's volume formula.
For the special case of arithmetic hyperbolic manifolds the same bound was established earlier in \cite{Be}.
We also discuss some applications, and give a straightforward proof for a weak version of the homotopy type conjecture from \cite{Ge}.
\end{abstract}


\section{Lower bound on the systole}
Let $G$ be a semisimple Lie group without compact factors and let $X=G/K$ be the associated symmetric space. We normalize the Haar measure of $G$ to be compatible with the Riemannian volume form of $X$.
Let $\gC\le G$ be a co-compact arithmetic group and $M=\gC\backslash X$ the corresponding compact $X$-orbifold. 
%
%
%

	\subsection{Embedding of $G/K$}
	Recall that $P(n, \mathbb{R})$, the space of all unimodular positive definite symmetric matrices of size $n\times n$, can be identified with the symmetric space $SL(n, \mathbb{R})/SO(n)$. This space is equipped with a metric $d$ such that for a point $x\in P(n, \mathbb{R})$ with eigenvalues $a_1, \dots, a_n$ the distance to identity is given by the following formula (cf. \cite{BH} Chapter II, Corollary 10.42).
	\begin{equation}
		d(I,x)=\sqrt{\sum_{i=1}^{n}\log^2(a_i)}.
		\label{eq:DistanceToID}
	\end{equation} 
	For a semisimple group $G$ the existence of Cartan involutions as described by G. Mostow (cf. \cite{Mos49}, Theorem 1) allows us to choose an appropriate scalar product $B$ on its Lie-algebra $\mathfrak{g}$ so that the map $\phi: G\rightarrow P(n, \mathbb{R}), g\mapsto Ad(g)Ad(g)^t$ provides a totally geodesic embedding of $G/K$ into $P(n,\mathbb{R})$ with $n=dim(\mathfrak{g})$. Any $g\in G$ acts on $\phi(G)$ by the isometry 
	$$
	I_g:\phi(G)\rightarrow\phi(G), x\mapsto Ad(g)xAd(g)^t.
	$$
	
	For a monic polynomial $f=\prod_{i} (x-a_i)$ the (exponential) Mahler measure $M(f)$ is defined by $M(f)=\prod_{\lvert a_i\rvert>1}\lvert a_i\rvert$. Originally the Mahler measure was defined for polynomials over $\mathbb{Q}$ with this definition, but we may extend to polynomials with complex coefficients. 
	Abusing notation, we set $M(g)$ to be the Mahler measure of the characteristic polynomial of $Ad(g)$ for $g\in G$.
	Now we can make the connection between Mahler measure $M(g)$ and the translation distance of $g$.
	\begin{prop}
		For any $g\in G$ the following holds:
		\begin{equation}
			\inf_{x\in P(n, \mathbb{R})}d(x,I_g(x))\geq \frac{2\log(M(g))}{n}.
			\label{eq:TranslationByMahler}
		\end{equation}
	\end{prop}
	\begin{proof}
		Since $x$ is a positive definite symmetric matrix we can find a positive definite symmetric matrix $\sqrt{x}$, whose square is $x$. Thus we have $d(x, I_g(x))= d(I, (\sqrt{x}^{-1}g\sqrt{x})(\sqrt{x}^{-1}g\sqrt{x})^t)$. 

		Now for  $a\in SL(n,\mathbb{R})$ the matrix $a a^t$ is symmetric and thus its maximal eigenvalue is given by
		\begin{equation}
			\max_{v\in V, \lvert v\rvert=1} (v^ta)(a^tv)
		\end{equation} 
		This shows that if $\lambda$ is the absolute of the largest eigenvalue of $a$ then $a a^t$ posses an eigenvalue larger than $\lambda^2$.
		Thus Inequality (\ref{eq:TranslationByMahler}) follows from Equation (\ref{eq:DistanceToID}). 
	\end{proof}
	We note that the bound is not optimal, for a better bound see \cite{LLM} which has the same asymptotics but improved coefficients.
	
	We denote by $Syst_1(M)$ the systole of $M$, that is the length of a shortest non-trivial closed geodesic in $M$.
	An element $g\in SL(n,\mathbb{R})$ which corresponds to a hyperbolic isometry of $P(n, \mathbb{R})$ must have at least one eigenvalue whose absolute value is larger than 1. Therefore: 
	\begin{cor}
		The systole of an arithmetic manifold $M=\Gamma\backslash G/K$ is bounded from below by:
		\begin{equation}
			Syst_1(M)\geq inf_{\gamma\in \Gamma, M(\gamma)>1}\frac{2\log(M(\gamma))}{n}.
			\label{eq:SystoleByMahler}
		\end{equation}
		\label{Cor:SystoleByEmbedding}
	\end{cor}
	\subsection{Degree of Definition and Dobrowolski's bound}
	Proving that the systole of an arithmetic manifold is bounded by its volume consists of two steps: First one bounds the systole by means of the algebraic data defining the arithmetic lattice and then applying Prasad's volume formula to express the lower bound for the systole in terms of the volume. To this end we now prove that we can find for each arithmetic lattice an effectively represented integral matrix group which surjects onto the lattice. 
	
	Recall that the isometry group of a simply connected non-positively curved manifold has no compact normal subgroups (\cite{Hel01} Chapter IV, Theorem 3.3). By the definition of arithmeticity there is for any arithmetic subgroup $\Gamma<G$, a connected algebraic group $H$ defined over $\mathbb{Q}$ and a compact normal subgroup $L\lhd H$ such that $H/L$ is isomorphic to $G$. Further the image of $H(\mathbb{Z})$ under this isomorphism is commensurable with $\Gamma$. The adjoint map is defined over $\mathbb{Z}$ (\cite{Vin71} Lemma 8) and  $G$, being the group of isometries of a non-positively curved symmetric space, is centre-free, so we can replace both $H$ and $G$ by their adjoint groups and use the corresponding projection $p:Ad(H)\rightarrow Ad(H)/Ad_H(L)$ and isomorphism $\phi:Ad(H)/Ad_H(L)\rightarrow Ad(G)$.
	
	Observe that since $H$ and $G$ are connected we can find a matrix $f \in SL(\mathfrak{g}, \mathbb{R})$ such that $f$ induces $\phi$ by conjugation. Further, since the kernel is compact, so $h\in H$ and $p(h)$ share the same eigenvalues of absolute different to $1$, which implies $M(h)=M(p(h))$. So we can conclude that for all $h\in H$, $M(h)=M(\phi\circ p(h))$. This observation is important as the lower bound for the systole will be provided by the so called Dobrowolski's bound on the Mahler measure:
	\begin{thm}
		(\cite{Dob79}) Given a non-cyclotomic integral monic polynomial $f=\prod_{i=1}^d (x-a_i)$ of degree $d$, the Mahler measure is bounded away from $1$ by its degree as in inequality 
		\begin{equation}
			\log(M(f))\geq \tilde{c}\left(\frac{\log\log d}{\log d}\right)^3,
			\label{eq:Dobrowolski}
		\end{equation}
	where $\ti c$ is a universal constant.		
	\end{thm}
	Up to a change of basis over $\mathbb{Q}$ we can assume that $\Gamma\subset p(H(\mathbb{Z}))$ (\cite{PR94}). As noted before  $M(p(h))=M(h)$ and thus the Mahler measure of $\gamma\in \Gamma$ is the Mahler measure of some integral matrix i.e. the premiage of $\gamma$ in $H(\mathbb{Z})$. The characteristic polynomial of $h\in H$ has degree $dim(H)$ as $H$ is adjoint. In order to apply Dobrowolski's bound we need to calculate $dim(H)$. $H$ is semi-simple so we can decompose it into a product fo simple parts. Each simple part is an algebraic group in itself defined over a number field.  For simplicity of notation we may assume that $\Gamma$ is irreducible and then all simple parts of $H$ must be Galois conjugates of each other. This means that there is a finite extension $k$ and each simple part is defined over a corresponding embedding of $k$ into $\mathbb{C}$. We call $k$ the field of definition for $\Gamma$ with degree $s=[k:\mathbb{Q}]$. As $H$ and $G$ are adjoint $dim(H)=s\cdot dim(G)$ holds. In conclusion we can find for any $\gamma\in \Gamma$ an integral polynomial of degree $s\cdot dim (G)$ which has the same Mahler measure as the characteristic polynomial of $\gamma$. Therefore, we may apply Dobrowolski's bound to corollary \ref{Cor:SystoleByEmbedding} and obtain:
	\begin{prop}\label{thm:systol}
		The systole of an arithmetic manifold $M=\Gamma\backslash G/K$, where $\Gamma$ has degree of definition $s$, is bounded from below by the following inequality: 
		\begin{equation}
				Syst_1(M)\geq c'\left(\frac{\log\log s}{\log s}\right)^3
				\label{eq:SystoleByDegree},
			\end{equation}
						for some absolute constant $c'$.
	\label{prop:SystoleByDegree}
	\end{prop}
	
Proposition \ref{thm:systol} allows us to establish:

\begin{thm}\label{prop:main}\label{thm:main}
		The systole of an arithmetic manifold is bounded from below in terms of its volume as in the following equation:
		\begin{equation}
			Syst_1(M)\geq C_1\left(\frac{\log\log\log \vol(M)^C}{\log\log \vol(M)^C}\right)^3
			\label{eq:SystoleByVolume}
		\end{equation}
		where $C,C_1$ are constants depending only on $G$.
	\end{thm}
	
	For $G=\text{PO}(n,1),~n\ge 2$, Theorem \ref{prop:main} was proved by M. Belolipetsky in \cite{Be10,Be}. Our proof is essentially the same.

	\begin{proof}
		M. Belolipetsky extracted the following bound for the co-volume of an arithmetic lattice  from  
		Prasad's volume formula in terms of the discriminant $\mathcal{D}_k$ (see \cite{Bel07} \S 3.3): 
		\begin{equation}
			covol(\Gamma)\geq d_k^{\delta_1}d_{l/k}^{\delta_2}/2
		\end{equation} 
		where $\delta_i>0$, $d$ is the absolute value of the discriminant $\mathcal{D}$ and $l$ is a field extension of $k$ with $[l:k]\leq 3$. Applying Minkowski's bound for the discriminant, we first may ignore $d_{l/kj}$ as it is larger than $1$ and then bound $d_{k}$ from below by an exponential function in $[k:\mathbb{Q}]$ (cf. \cite{Lan86}). 
		Taking a logarithm this yields
		\begin{equation}
			s\leq c_1 log(covol(\Gamma))+c_2\leq c_3 log(covol(\Gamma))
			\label{eq:SuperVolume}
		\end{equation}
		where $c_i$ depends on $G$ and the last inequality is implied by Kazhdan--Margulis theorem bounding the covolume of arithmetic lattices from below. Substituting Inequality \ref{eq:SuperVolume} into Inequality \ref{eq:SystoleByDegree} we obtain Theorem \ref{prop:main}.
	\end{proof}

%
%

%
%
%


\section{The homotopy type conjecture}

Let $d\in \BN$ and $v>0$. A $(d,v)$-simplicial complex is a simplicial complex with at most $v$ vertices such that all the vertices are of degree  bounded by $d$. The parameter $d$ is related to dimension where $v$ is related to volume. Note that the number of $k$-simplices of a $(d,v)$-simplicial complex is at most $v\cdot {d \choose k}/(k+1)$ which is linear in terms of $v$. 

\begin{defn}\label{def:HC}
Let $\mathcal{F}$ be a family of $d$-dimensional Riemannian manifolds. We say that the family $\mathcal{F}$ has {\it uniform homotopy complexity} if there are constants $D,\ga$ such that every $M\in \mathcal{F}$ is homotopy equivalent to some  $(D,\ga\cdot\vol(M))$-simplicial complex $\mathcal{R}$.
\end{defn}

The homotopy type of aspherical manifolds is determined by the fundamental group which is often easier to control. This is the reason to make the following definition.

\begin{defn}\label{def:pi_1}
Let $\mathcal{F}$ be a family of $d$-dimensional aspherical Riemannian manifolds. We say that the family $\mathcal{F}$ has {\it uniform $\pi_1$-complexity} if there are constants $D,\ga$ such that every $M\in \mathcal{F}$ satisfies $\pi_1(M)\cong\pi_1(\mathcal{R})$
for some  $(D,\ga\cdot\vol(M))$-simplicial complex $\mathcal{R}$.
\end{defn}

Note that the simplicial complex $\mathcal{R}$ in Definition \ref{def:pi_1} is not required to be aspherical. 

Recall the following conjecture from \cite{Ge}:

\begin{conj}\label{conj:HTC}
Let $X$ be a symmetric space of non-compact type. Then the family of arithmetic $X$-manifolds has uniform homotopy complexity. 
\end{conj}

For the special case $X=\mathcal{H}^2$,  Conjecture \ref{conj:HTC} is a consequence of the Gauss--Bonnet theorem.
Conjecture \ref{conj:HTC} was proved in \cite{Ge} for non-uniform arithmetic manifolds:

\begin{thm}[\cite{Ge}, Theorem 1.5(i)]\label{thm:hv1}
Let $X$ be a symmetric space of non-compact type. The family of all non-uniform arithmetic $X$-manifolds has uniform homotopy complexity.
\end{thm}

Concerning $\pi_1$-complexity, the following result was obtained in \cite{Ge} for all symmetric spaces of non-compact type with a few exceptions:

\begin{thm}[\cite{Ge}, Theorem 1.5(ii)]\label{thm:hv2}
Let $X$ be a symmetric space of non-compact type. Suppose that $X$ is not isometric to $\mathcal{H}^3,~\mathcal{H}^2\times\mathcal{H}^2$ or $\SL_3(\BR)/\SO(3)$. Then the family of all irreducible $X$-manifolds has uniform $\pi_1$-complexity.
\end{thm}

Note that Theorem \ref{thm:hv2} does not require arithmeticity. 

In a recent\footnote{The preprint of the current paper was written in 2021.} remarkable paper \cite{Fr}, M. Fraczyk settled Conjecture \ref{conj:HTC} for arithmetic $3$-manifolds.

\begin{thm}[Fraczyk, \cite{Fr}]\label{thm:Fraczyk}
The family of compact hyperbolic arithmetic $3$-manifolds has uniform homotopy complexity.
\end{thm}

Very recently, Conjecture \ref{conj:HTC} for compact arithmetic manifolds has been settled in general by M. Fraczyk, S. Hurtado and  J. Raimbault \cite{FHR}.
We remark that their proof makes use of Theorem \ref{thm:main} above.

We will now show that a slightly weaker version of conjecture \ref{conj:HTC} follows quite immediately from Theorem \ref{prop:main}.

\begin{thm}\label{thm:WHT}
Let $X$ be a symmetric space of non-compact type of dimension $d=\dim(X)$.
Then there are constants $D,\ga$ such that any arithmetic $X$-manifold is homotopy equivalent to some $(D,\ga\cdot\phi(v)v)$-simplicial complex, where $v=\vol(M)$ and $\phi$ is the (slowly growing) function:
$$
 \phi(v):=\left(\frac{\log\log(v)}{\log\log\log(v)}\right)^{3d}.
$$
\end{thm}



In view of Theorem \ref{thm:hv1} we may restrict our attention to compact arithmetic $X$-manifolds. Consider such a manifold $M$ and let $\mathfrak{s}=\min\{ Syst_1(M), 1\}$ be the systole of $M$ or $1$ if the systole is greater. By Theorem \ref{prop:main},
$$
 \mathfrak{s}\geq C_1\left(\frac{\log\log\log \vol(M)^C}{\log\log \vol(M)^C}\right)^3
$$
where $C,C_1$ are constants depending only on $X$.

Let $F$ be a maximal $\mathfrak{s}/2$-discrete subset of $M$. Then by maximality the collection $\mathcal{U}$ of $\mathfrak{s}/2$-balls centred at the points of $F$ forms a cover of $M$. Furthermore any non-empty intersection of elements of $\mathcal{U}$ is intrinsically convex and hence contractible. It follows that $\mathcal{U}$ forms a good cover of $M$. Thus, letting $\mathcal{R}=\mathcal{R}(\mathcal{U})$ be the nerve of $\mathcal{U}$, it follows from the nerve lemma (c.f. \cite[\S 2]{BGS2}) that $M$ is homotopy equivalent $\mathcal{R}$. Now since the $\mathfrak{s}/4$-balls centred at $F$ are disjoint, it follows that 
$|F|\le\vol(M)/\vol(B(\mathfrak{s}/4))$ where $B(\mathfrak{s}/4)$ is a ball of radius $\mathfrak{s}/4$ in $X$. Note that the volume of such a ball is bounded from below by the volume of an $\mathfrak{s}/4$-ball in $\BR^d$. Moreover, as the distance between any two points in $F$ which correspond to adjacent vertices in $\mathcal{R}$ is at most $\mathfrak{s}$, it follows that the degree of every vertex at $\mathcal{R}$ is at most $\vol(B(1.25\mathfrak{s}))/\vol(B(0.25\mathfrak{s}))$. Thus, Theorem \ref{thm:main} follows from the following:

\begin{lem}
There is a constant $D=D(d)$ such that whenever
$X$ is a Hadamard manifold of curvature bounded from below by $-1$ and dimension $d$ and $x_0\in X$, we have  
$$
\sup_{r\le 1} \vol(B(x_0,1.25r))/\vol(B(x_0,0.25r))\le D.
$$

\end{lem} 

The last lemma follows by combining the Bishop--Gunther (lower bound) and the Bishop--Gromov (upper bound) comparison theorems (see \cite[Theorem~3.101 on p.~169 and Theorem~4.19 on p.~214]{GHL}):

\begin{thm}\label{thm:comparison}
Let $X$ be a Hadamard manifold of dimension $d$ and sectional curvature varying in the segment $[-1,0]$. For every  $R>0$ and every $x\in X$ we have
\[ \vol_{\mathbb{E}^d}(B(0,R)) \leq \vol_X(B(x,R)) \leq \vol_{\mathcal{H}^d}B(o,R). \]
\end{thm}

Indeed, in view of Theorem \ref{thm:comparison}, $D$ is bounded from above by the volume of a $1.25$-ball in $\mathcal{H}^d$ divided by the volume of a $0.25$-ball in $\mathbb{E}^d$. 
This completes the proof of Theorem \ref{thm:WHT}.
\qed


\section{Bounds on torsion}

As demonstrated in \cite{Ge}, many homotopy invariants can be bounded in terms of the $\pi_1$-complexity (see for instance Theorem 1.7 and Theorem 1.11 in \cite{Ge}). 
One exception for this phenomenon is the torsion of the integral homology $\text{tors} H_k(M;\BZ)$.\footnote{Recall that the linear bound on the Betti number is a classical theorem of Gromov \cite[Theorem 2]{BGS} which holds in the much greater generality of analytic manifolds of a fixed dimension of bounded non-positive curvature and without Euclidean factors for the universal cover.}
Conjecture \ref{conj:HTC} implies that: 
\begin{equation}\label{linear-torsion}
 \log  |\text{tors} H_k(M;\BZ)|\le C\cdot\vol(M),
\end{equation}
for some constant $C$ depending on the universal cover. In view of Theorem \ref{thm:hv1}, the conjectural estimate (\ref{linear-torsion}) holds for non-uniform arithmetic manifolds.
Moreover, the estimate (\ref{linear-torsion}) was proved in \cite{BGS2} for the general class of negatively curved manifolds of a fixed dimension $d\ne 3$ and curvature bounded in a closed subinterval of $[-1,0)$ (the constant $C$ depends $d$ but not on the upper bound of the curvature). In particular this result applies to all locally symmetric spaces of rank one and dimension other than $3$. In dimension $3$ the analog result is not true for nonarithmetic manifolds (see \cite{Gromov}) while in view of Theorem \ref{thm:Fraczyk} it is true for the class of $3$-dimensional arithmetic manifolds (and the recent preprint \cite{FHR} confirms the case of compact arithmetic manifolds of general type). 

In view of \cite[Lemma 5.2]{BGS2}, the following weaker (but almost as good) estimate is an immediate consequence of Theorem \ref{thm:main}:

\begin{cor}
Let $X$ be a symmetric space of non-compact type of dimension $d$. There is a constant $C=C(X)$ such that for every compact arithmetic $X$-manifold, $M$, we have:
$$
 \log  |\text{tors} H_k(M;\BZ)|\le C \left(\frac{\log\log(\vol(M))}{\log\log\log(\vol(M))}\right)^{3d}\vol(M).
$$
\end{cor}

%

\end{document}